\documentclass[a4paper]{article}

\usepackage[utf8]{inputenc}
\usepackage{lmodern}
\usepackage{array}
\usepackage{graphicx}
\usepackage{stmaryrd}
\usepackage{subfig}
\usepackage{amssymb, amsfonts, amsthm, amsmath}
\usepackage{enumerate}
\usepackage{enumitem}
\usepackage{bbm}
\usepackage[english]{babel}
\usepackage{tikz}
\usetikzlibrary{decorations}
\usetikzlibrary{decorations.pathreplacing}
\tikzset{individu/.style={draw,thick}}

\numberwithin{equation}{section}

\theoremstyle{plain}
\newtheorem{theorem}{Theorem}[section]
\newtheorem{corollary}[theorem]{Corollary}

\newtheorem{lemma}[theorem]{Lemma}
\newtheorem{proposition}[theorem]{Proposition}

\theoremstyle{definition}

\theoremstyle{remark}

\newcommand{\Z}{\mathbb{Z}}
\newcommand{\R}{\mathbb{R}}

\newcommand{\cc}{\mathbb{C}}
\newcommand{\cp}{\mathbb{CP}}
\newcommand{\Ga}{\Gamma}

\newcommand{\calC}{\mathcal{C}}

\newcommand{\calS}{\mathcal{S}}

\newcommand{\calH}{\mathcal{H}}

\newcommand{\calT}{\mathcal{T}}
\newcommand{\calG}{\mathcal{G}}

\newcommand{\calA}{\mathcal{A}}
\newcommand{\calD}{\mathcal{D}}

\def\wt#1{\widetilde #1}

\def\cc{\mathbb C}
\def\oc{\overline\cc}
\def\cp{\mathbb{CP}}

\def\mcd{\mathcal D}

\renewcommand{\epsilon}{\varepsilon}
\renewcommand{\phi}{\varphi}

\title{A first integrability result for Miquel dynamics}

\author{Alexey Glutsyuk\thanks{ CNRS, France (UMR 5669 (UMPA, ENS de Lyon) and
Interdisciplinary Scientific Center J.-V.Poncelet), 
Lyon, France. 
E-mail:
aglutsyu@ens-lyon.fr}
\thanks{National Research University Higher School of Economics (HSE),
Moscow, Russia}
\thanks{Supported by part by RFBR grants 16-01-00748, 16-01-00766} \and Sanjay Ramassamy\thanks{Unit\'e de Math\'ematiques Pures et Appliqu\'ees, \'Ecole normale sup\'erieure de Lyon, 46 all\'ee d'Italie, 69364 Lyon Cedex 07, France. 
E-mail:
sanjay.ramassamy@ens-lyon.fr}
\thanks{Supported by the Fondation Simone et Cino Del Duca}
}

\date{\today}

\begin{document}

\maketitle

\begin{abstract}
Miquel dynamics is a discrete-time dynamical system on the space of square-grid circle patterns. For biperiodic circle patterns with both periods equal to two, we show that the dynamics corresponds to translation on an elliptic curve, thus providing the first integrability result for this dynamics. The main tool is a geometric interpretation of the addition law on the normalization of binodal quartic curves.
\end{abstract}

\section{Introduction}

Miquel dynamics was introduced by the second author in~\cite{R17}, following an original idea of Richard Kenyon~\cite{K14}, as a discrete-time dynamical system on the space of square-grid circle patterns. It was then conjectured that for biperiodic circle patterns, Miquel dynamics belongs to the class of discrete integrable systems, which contains among others the dimer model~\cite{GK} and the pentagram map~\cite{OST1,OST2,S13}. In this article, we show that in the particular case when both periods are equal to two, Miquel dynamics corresponds, in the right coordinates, to translation on an elliptic curve. This is the first integrability result established for Miquel dynamics. An important observation we make to prove this is a simple geometric interpretation of the addition law on the normalization of algebraic curves of degree four with two nodes.

\subsection{Circle patterns and Miquel dynamics}

A \emph{square grid circle pattern} (abbreviated as SGCP) is a collection of points $(S_{i,j})_{(i,j)\in\Z^2}$ in the plane $\R^2$ such that for any $(i,j)\in\Z^2$, the points $S_{i,j},S_{i+1,j}$, $S_{i,j+1} $ and $S_{i+1,j+1}$ are pairwise distinct and concyclic, with the circle going through them denoted by $C_{i,j}$. The circles are colored in a checkerboard pattern: the circles $C_{i,j}$ with $i+j$ even (resp. odd) are colored black (resp. white). The center of the circle $C_{i,j}$ is denoted by $O_{i,j}$. We define two maps $\mu_w$ and $\mu_b$, respectively called \emph{white mutation} and \emph{black mutation}, from the set of SGCPs to itself. For any SGCP $S$, the SGCP $T:=\mu_w(S)$ is obtained as follows: for any $(i,j)\in\Z^2$ such that $i+j$ is even (resp. odd), $T_{i,j}$ is obtained by reflecting $S_{i,j}$ through the line $(O_{i,j}O_{i-1,j-1})$. (resp. $(O_{i-1,j}O_{i,j-1})$). It follows from Miquel's six-circles theorem~\cite{M38} that $T$ is indeed a circle pattern, with the same black circles as $S$ but with potentially different white circles. Similarly, for any SGCP $S$, the SGCP $T':=\mu_b(S)$ is obtained as follows: for any $(i,j)\in\Z^2$ such that $i+j$ is even (resp. odd), $T'_{i,j}$ is obtained by reflecting $S_{i,j}$ through the line $(O_{i-1,j}O_{i,j-1})$. (resp. $(O_{i,j}O_{i-1,j-1})$). Each mutation is an involution. Miquel dynamics is defined as the discrete-time dynamical system obtained by applying alternately $\mu_w$ followed by $\mu_b$. Note that this dynamics is different from the one on circle configurations studied by Bazhanov, Mangazeev and Sergeev~\cite{BMS}, which uses a different version of Miquel's theorem.

Given two positive even integers $m$ and $n$ and two non-collinear vectors $\vec{u}$ and $\vec{v}$ in $\R^2$, an SGCP $S$ is said to be $(m,n)$-\emph{biperiodic} with monodromies $\vec{u}$ and $\vec{v}$ if for any $(i,j)\in\Z^2$, the following two conditions hold:
\begin{enumerate}
 \item $S_{i+m,j}=S_{i,j}+\vec{u}$ ;
 \item $S_{i,j+n}=S_{i,j}+\vec{v}$.
\end{enumerate}

We denote by $\calS_{m,n}$ the set of all $(m,n)$-biperiodic SGCPs (with arbitrary monodromies). This set is stable under both black mutation and white mutation. Miquel dynamics on $\calS_{m,n}$ is conjectured to be integrable in some sense. In this paper we provide a first integrability result in the case when $m=n=2$. For the remainder of the paper, all SGCPs will be in $\calS_{2,2}$.

Let $S\in\calS_{2,2}$ be an SGCP. We will denote its vertices in the fundamental domain $\left\{0,1,2\right\}^2$ as follows (see Figure~\ref{fig:pattern} for an illustration):
\[
\begin{array}{ccc}
A=S_{0,0} & B=S_{1,0} & C=S_{2,0} \\
D=S_{0,1} & E=S_{1,1} & F=S_{2,1} \\
G=S_{0,2} & H=S_{1,2} & I=S_{2,2}
\end{array}
\]

\begin{figure}[htbp]
\centering
\includegraphics[height=2.5in]{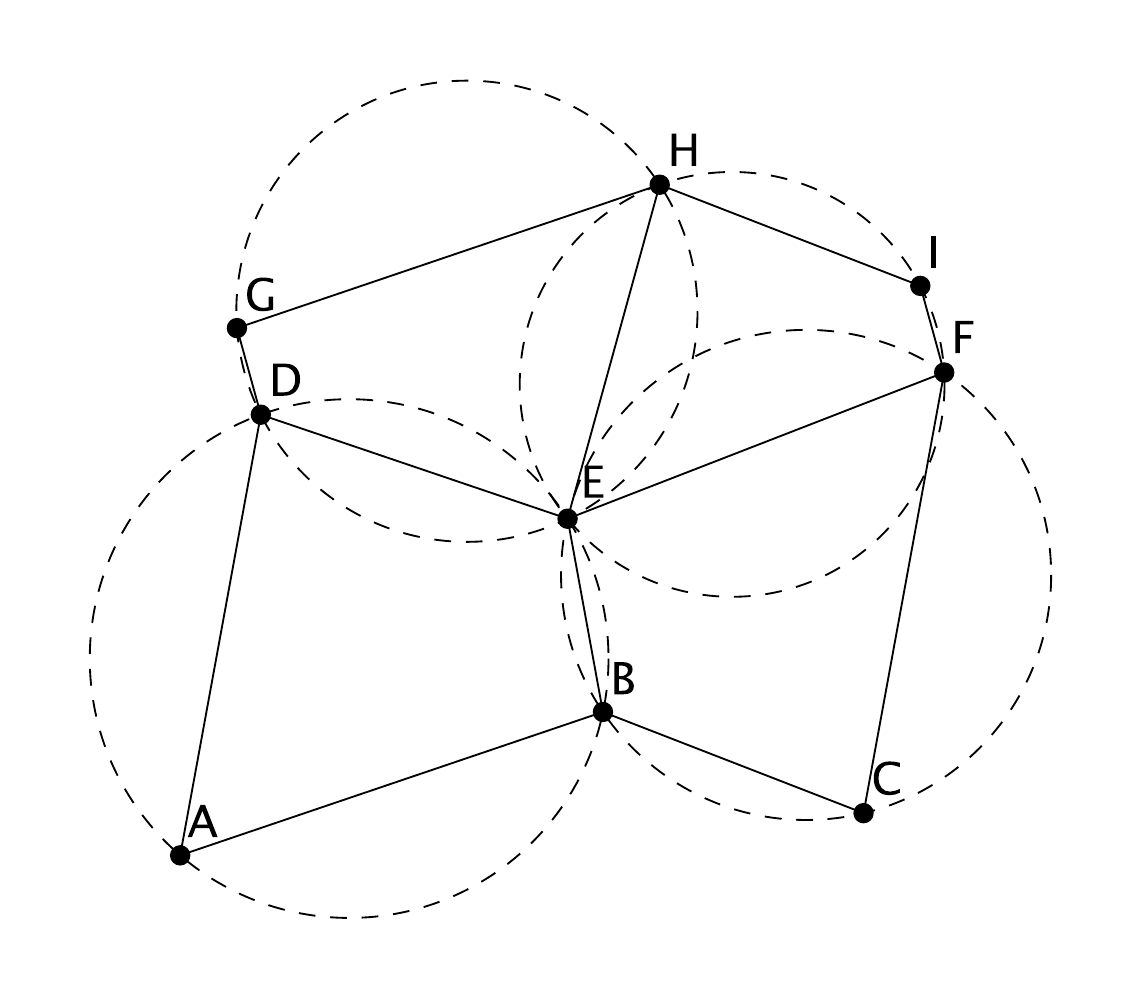}
\caption{Illustration of the notation for a biperiodic circle pattern with both periods equal to two.}
\label{fig:pattern}
\end{figure}

Set $S_w:=\mu_w(S)$ and $S_b:=\mu_b(S)$. We will denote their vertices in the fundamental domain $\left\{0,1,2\right\}^2$ respectively by $A_w,\ldots,I_w$ and $A_b,\ldots,I_b$. Instead of looking at the absolute motion of the points, we consider the relative motion of points with respect to one another. To do so, we introduce the pattern $S_w'$ (resp. $S_b'$) which is obtained from $S_w$ (resp. $S_b$) by applying the translation of vector $\overrightarrow{A_wA}$ (resp. $\overrightarrow{A_bA}$). We call \emph{renormalized white mutation} $\mu_w'$ (resp. \emph{renormalized black mutation} $\mu_b'$) the map which to $S$ associates $S_w'$ (resp. $S_b'$). We denote the vertices of $S_w'$ and $S_b'$  in the fundamental domain $\left\{0,1,2\right\}^2$ respectively by $A_w',\ldots,I_w'$ and $A_b',\ldots,I_b'$. It was shown in~\cite{R17} that both points $E_w'$ and $E_b'$ lie on some explicit quartic curve $Q_S$, which also contains the points $A,C,E,G$ and $I$ (see Section~\ref{sec:reminder} for a precise definition of $Q_S$). In other words, the relative motion of the point in position $(1,1)$ with respect to the point in position $(0,0)$ lies on this curve $Q_S$. The curve $Q_S$ has, in an appropriate coordinate system, an equation of the form
\begin{equation}
\label{eq:Miquelquartic}
(x^2+y^2)^2+ax^2+by^2+c=0,
\end{equation}
with $(a,b,c)\in\mathbb{R}^3$. See Figure~\ref{fig:quartic} for an example. We call a \emph{Miquel quartic} a quartic curve which has an equation of the form~\eqref{eq:Miquelquartic}. As a special case of Miquel quartics, when $a+b=0$, we obtain the family of Cassini ovals~\cite{Y47}.

\begin{figure}[htbp]
\centering
\includegraphics[height=2in]{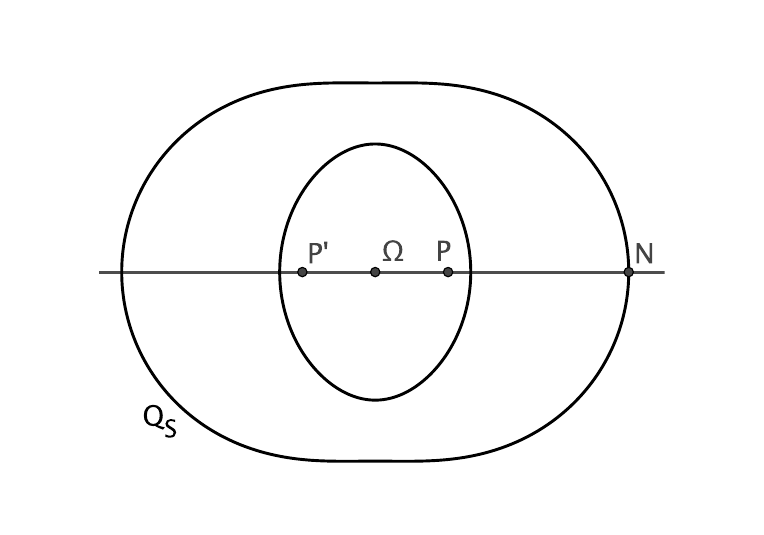}
\caption{Example of a quartic curve $Q_S$, with the center $\Omega$, the foci $P$ and $P'$ and the neutral element for addition $N$ all lying on the horizontal coordinate axis. The curve $Q_S$ may have either one or two ovals.}
\label{fig:quartic}
\end{figure}

\subsection{Addition on binodal quartic curves}

A complex quartic curve in $\cp^2$ is called \emph{binodal} if it has two nodes, i.e. singularities at which two regular local branches intersect transversally. A binodal quartic curve is called \emph{non-degenerate} if it has no other singularities. The projective closure in $\cp^2$ of a Miquel quartic is generically a non-degenerate binodal quartic curve, with its two nodes being the \emph{circular points at infinity} (also called isotropic points) with homogeneous coordinates $(1:\pm i:0)$, which lie on the infinity line $\oc_{\infty}=\cp^2\setminus\cc^2$ (see also Lemma~\ref{lem:quarticdegeneration} for a more precise statement). Every non-degenerate binodal quartic has an elliptic normalization, that is a holomorphic parametrization by an elliptic curve that is bijective outside the nodes. Thus, its normalization has a natural group structure, which is unique once the neutral element is chosen. Everywhere below, whenever we write about the addition law on a non-degenerate binodal quartic, we mean the addition law on its elliptic normalization. By an abuse of notation, the points on a non-degenerate binodal quartic and their lifts to the elliptic normalization will be denoted by the same symbols. 

We will show that, for Miquel dynamics, the motion of the point in position $(1,1)$ induced by the composition of a black and a white renormalized mutation corresponds to an explicit translation on the normalization of $Q_S$. Another geometric interpretation of translation on certain binodal quartic curves, in terms of foldings of quadrilaterals, was introduced by Darboux~\cite{D79} and studied more recently in~\cite{BH04,I15}. In order to state our translation result, we first give a geometric interpretation of the addition on non-degenerate binodal quartic curves.

The elliptic curve group law has a well-known geometric interpretation in the case of smooth cubic curves. In recent years, other geometric interpretations have arisen in the case of quartic curves, such as Edwards curves, later generalized to twisted Edwards curves~\cite{E07,BBJLP,H16}, and Jacobi quartic curves~\cite{W11}. For a non-degenerate binodal quartic curve, we obtain the addition law by fixing an arbitrary base point and declaring that, whenever a conic passes through both nodes and the base point, the other three intersection points of the conic with the quartic have zero sum. More specifically, we have the following theorem. 

\begin{theorem}
\label{thm:quarticaddition}
Let $\Ga\subset\cp^2$ be a non-degenerate binodal quartic curve with nodes $T_1$ and $T_2$, and let $P\in\Ga\setminus\{ T_1,T_2\}$ be a base point. Consider the two-dimensional family $\calC_P$ of conics through the three points $T_1$, $T_2$, $P$. Any conic $c\in \calC_P$ intersects $\Ga$ at three additional points $X_c,Y_c$ and $Z_c$, which need not be distinct and may coincide with $T_1$, $T_2$ or $P$. 
Then one can choose a neutral element $N_P$ for the addition law on the normalization $\hat\Ga$ of $\Ga$ such that $X_c+Y_c+Z_c=0$ for every $c\in\calC_P$.  
\end{theorem}

Theorem~\ref{thm:quarticaddition} easily follows from the classical theory of adjoint curves (see for example~\cite[section 49]{V45}). However, it does not seem to be explicitly stated in the literature. This result generalizes the geometric interpretation of addition for twisted Edwards curves in terms of intersections with hyperbolas~\cite{H16}.

We use the above theorem to construct the group law on Miquel quartics. We denote by $\calS_{2,2}^0\subset\calS_{2,2}$ the subset of patterns $S$ such that the binodal quartic $Q_S$ is non-degenerate. We will see in Lemma~\ref{lem:quarticdegeneration} that a Miquel quartic with an equation of the form~\eqref{eq:Miquelquartic} is non-degenerate if and only if $4c\notin \left\{0,a^2,b^2\right\}$ and $a \neq b$, so that $\calS_{2,2}^0$ is Zariski-open in $\calS_{2,2}$.

\begin{proposition}
\label{prop:miquelgrouplaw}
Let $S\in\calS_{2,2}^0$. We pick $N$ to be an intersection point of $Q_S$ with the $x$-axis. One can consider the group law on $Q_S$ as constructed in Theorem~\ref{thm:quarticaddition} with $N$ being both the base point and the neutral element. The inverse is given by reflection through the $x$-axis. The sum of two points $P_1$ and $P_2$ is given by taking the circle going through $P_1,P_2$ and $N$ and reflecting through the $x$-axis the fourth point of intersection $P_3$ of this circle with $Q_S$.
\end{proposition}

\begin{theorem}
\label{thm:translation}
Using the group law on $Q_S$ defined in Proposition~\ref{prop:miquelgrouplaw} and the notation for $A,C,E,E_w'$ and $E_b'$ defined above, we have for any $S\in\calS_{2,2}^0$
\begin{align}
E_w'&=-E-2A \label{eq:whitemutation} \\
E_b'&=-E-2C \label{eq:blackmutation}.
\end{align}
In particular, the composition of a renormalized white mutation followed by a renormalized black mutation produces a translation by $2(A-C)$.
\end{theorem}

It follows from~\cite{R17} that the space $\calS_{2,2}$ is of real dimension $9$. It was also shown there that, after application of the composition $\mu_w' \circ \mu_b'$ of renormalized black and white mutations, the point $A$, the vectors $\overrightarrow{AC}$ and $\overrightarrow{AG}$ as well as the angles $\angle CBA$ and $\angle ADG$ do not change (see also Subsection~\ref{subsec:conservedquantities}). These provide eight real conserved quantities. Once we fix the values of these eight conserved quantities, we obtain a one-parameter family of possible patterns, parametrized by the position of $E$ on a quartic curve of Miquel type. Theorem~\ref{thm:translation} shows that Miquel dynamics induces a translation on that quartic curve. This is a sign of integrability of Miquel dynamics for $(2,2)$-biperiodic circle patterns, with the quartic curve playing, in a sense, the role of a Liouville torus. It also reinforces the conjecture about the integrability of Miquel dynamics for general $(m,n)$-biperiodic circle patterns.

We conclude this introduction by deriving, as a consequence of Theorem~\ref{thm:translation}, a measure on Miquel quartics which is invariant under Miquel dynamics.

\begin{corollary}
\label{cor:invariantmeasure}
Let $Q$ be a non-degenerate Miquel quartic with an equation of the form~\eqref{eq:Miquelquartic}. The 1-form 
\begin{equation}
\label{eq:form}
\omega=\frac{d(x^2+y^2)}{xy}
\end{equation}
is invariant under any translation on $Q$. In particular, for any $S\in\calS_{2,2}^0$, the composition $\mu_b' \circ \mu_w'$ induces a map for the motion of $E$ on $Q_S$ which leaves invariant the form $\omega$ on $Q_S$. Furthermore, for such a map on $Q_S$, the modulus $|\omega|$ is an invariant measure.
\end{corollary}

\subsection*{Outline of the paper}
Section~\ref{sec:reminder} consists in recalling several results needed from~\cite{R17}: the connection between patterns in $\calS_{2,2}$ and five-pointed equilateral hyperbolas, the integrals of motion for Miquel dynamics on $\calS_{2,2}$ and the construction of the quartic curve $Q_S$. In Section~\ref{sec:tangentcircles}, we describe a simple geometric construction of $E_w'$ starting from the curve $Q_S$ and two points $A$ and $E$ on it. Finally in Section~\ref{sec:grouplaw}, we prove Theorem~\ref{thm:quarticaddition} about the group law on non-degenerate binodal quartics and, combining it with the geometric construction of the previous section, we prove Theorem~\ref{thm:translation} and Corollary~\ref{cor:invariantmeasure} about Miquel dynamics.

\section{The space $\calS_{2,2}$ and the curve $Q_S$}
\label{sec:reminder}

In this section we describe the dichotomy for patterns in $\calS_{2,2}$, between generic patterns and trapezoidal patterns. We provide the integrals of motion for renormalized mutations, which include a quartic curve. We explain how to construct this quartic curve in both the generic and trapezoidal cases. All the statements in this section were proved in~\cite{R17}.

\subsection{Generic and trapezoidal patterns}
\label{subsec:generictrap}

Let $S\in\calS_{2,2}$. Denote its vertices by $A,\ldots,I$ as on Figure~\ref{fig:pattern}. The pattern $S$ has four cyclic faces, $ABED$, $BCFE$, $DEHG$ and $EFIH$. Then either all the faces of $S$ are trapezoids (trapezoidal case) or no face of $S$ is a trapezoid (generic case). The trapezoidal case is subdivided into two subcases :
\begin{itemize}
 \item each triple of points $\left\{A,B,C\right\}$, $\left\{D,E,F\right\}$ and $\left\{G,H,I\right\}$ is aligned (horizontal trapezoidal case) ;
 \item each triple of points $\left\{A,D,G\right\}$, $\left\{B,E,H\right\}$ and $\left\{C,F,I\right\}$ is aligned (vertical trapezoidal case).
\end{itemize}
We denote respectively by $\calG$, $\calT_h$ and $\calT_v$ the classes of generic, horizontal trapezoidal and vertical trapezoidal patterns. It was shown that each of these three classes is stable under Miquel dynamics.

Patterns in $\calS_{2,2}$ enjoy a special property, formulated in terms of an equilateral hyperbola. Recall that an equilateral hyperbola is a hyperbola with orthogonal asymptotes. Degenerate cases of equilateral hyperbolas correspond to the union of two orthogonal lines.

\begin{proposition}[\cite{R17}]
\label{prop:hyperbola}
Fix $S\in\calS_{2,2}$. There exists an equilateral hyperbola $\calH$ going through the points $B,D,E,F$, and $H$. Furthermore, $\calH$ is non-degenerate if and only if $S\in\calG$. The pattern $S$ is in $\calT_h$ (resp. $\calT_v$) if and only if the points $D,E$ and $F$ (resp. $B,E$ and $H$) lie on one line of $\calH$ and the points $B$ and $H$ (resp. $D$ and $F$) lie on the other line of $\calH$.
\end{proposition}

It was actually shown in~\cite{R17} that the space $\calS_{2,2}$ is parametrized by five-pointed equilateral hyperbolas: pick an equilateral hyperbola $\calH$ (four degrees of freedom) then pick five distinct points $B,D,E,F,H$ on $\calH$, there is a unique way to reconstruct a pattern $S\in\calS_{2,2}$ from this data, inverting the construction of Proposition~\ref{prop:hyperbola}.

The vertical trapezoidal case is handled in a similar fashion as the horizontal trapezoidal case. Thus from now on, among the trapezoidal patterns we shall only consider the horizontal trapezoidal ones and refer to the horizontal trapezoidal case simply as the trapezoidal case, omitting ``horizontal''.

\subsection{Conserved quantities under renormalized mutations}
\label{subsec:conservedquantities}

We recall the results about the invariants under Miquel dynamics found in~\cite{R17}. Fix $S\in\calS_{2,2}$ and write $S_w'=\mu_w'(S)$ and $S_b'=\mu_b'(S)$. Denote by $A'_w,\ldots,I'_w$ (resp. $A'_b,\ldots,I'_b$) the vertices of $S'_w$ (resp. $S'_b$). The following statements hold :
\begin{itemize}
 \item $(A,C,G,I)=(A_w',C_w',G_w',I_w')=(A_b',C_b',G_b',I_b')$ ;
  \item $\angle CBA = - \angle C_w'B_w'A_w' = - \angle C_b'B_b'A_b'$ ;
  \item $\angle ADG = - \angle A_w'D_w'G_w' = - \angle A_b'D_b'G_b'$.
\end{itemize}
 
Furthermore, there exists a quartic curve $Q_S$, the construction of which is explained in the next two subsections, verifying the following properties.

\begin{proposition}[\cite{R17}]
\label{prop:quarticstability}
For any $S\in\calS_{2,2}$, we have $Q_S=Q_{S_w'}=Q_{S_b'}$. Furthermore, the points $A$, $C$, $G$, $I$, $E$, $E'_w$ and $E_b'$ lie on $Q_S$.
\end{proposition}

\subsection{Construction of $Q_S$ in the generic case}
\label{subsec:genericquartic}

Assume $S\in\calG$. In this case it was shown in~\cite{R17} that the angles $\angle CBA$ and $\angle ADG$ are not flat. Denote by $O_B$ (resp. $O_D$) the center of the circle through $A,B$ and $C$ (resp. $A,D$ and $G$). Let $P$ be the intersection point of the parallel to $(AG)$ through $O_B$ and the parallel to $(AC)$ through $O_D$. Let $P'$ be the symmetric of $P$ across $\Omega$, where $\Omega$ is the center of the parallelogram $ACIG$. Define
\[
\lambda:=\frac{PA^2P'A^2-PC^2P'C^2}{\Omega A^2-\Omega C^2}
\]
and
\[
k:=\frac{\Omega A^2PC^2P'C^2-\Omega C^2PA^2P'A^2}{\Omega A^2-\Omega C^2}
\]
Then we define $Q_S$ be to the following locus of points in $\R^2$ :
\begin{equation}
\label{eq:genericquartic}
Q_S:=\left\{M\in\R^2 | PM^2P'M^2-\lambda \Omega M^2 = k \right\}.
\end{equation}
The points $P$ and $P'$ are called the foci of the quartic $Q_S$ (see Figure~\ref{fig:quartic}). Taking coordinates centered at $\Omega$ and such that $P$ lies on the $x$-axis, we obtain for $Q_S$ an equation of the form~\eqref{eq:Miquelquartic}.

\subsection{Construction of $Q_S$ in the trapezoidal case}
\label{subsec:trapquartic}

Assume $S\in\calT_h$. Take coordinates centered at the center $\Omega$ of the parallelogram $ACIG$ (which is actually a rectangle in the trapezoidal case), with the $x$-axis parallel to $(AC)$. The points $C,D$ and $E$ have respective coordinates $(x_C,y_C),(x_D,y_E)$ and $(x_E,y_E)$. Define the quantities
\begin{align}
\alpha &= x_C^2+y_C^2+x_E^2+y_E^2+\frac{(x_D+x_C)^2(x_C^2+y_C^2-x_E^2-y_E^2)}{y_C^2-y_E^2} \label{eq:alpha} \\
\beta &= x_C^2+y_C^2+x_E^2+y_E^2+\frac{(x_D+x_C)^2(x_E^2-x_C^2)(x_C^2+y_C^2-x_E^2-y_E^2)}{(y_C^2-y_E^2)^2} \label{eq:beta} \\
\gamma &= (x_C^2+y_C^2)(x_E^2+y_E^2)+\frac{(x_D+x_C)^2(x_E^2y_C^2-x_C^2y_E^2)(x_C^2+y_C^2-x_E^2-y_E^2)}{(y_C^2-y_E^2)^2}
\end{align}
Then $Q_S$ is the curve of equation
\begin{equation}
\label{eq:nongenericquartic}
(x^2+y^2)^2-\alpha x^2-\beta y^2+\gamma=0.
\end{equation}
It would be interesting to have a coordinate-free geometric construction of $Q_S$ in the trapezoidal case, as we had in the generic case.

\section{Another construction of renormalized mutation}
\label{sec:tangentcircles}

The proof of statement~\eqref{eq:whitemutation} in Theorem~\ref{thm:translation} relies on a direct construction of $E_w'$ from the points $A,E,I$ and the quartic curve $Q_S$. Denote by $O_A$ (resp. $O_I$) the center of the circle $\calC_A$ (resp. $\calC_I$) going through $A$ (resp. $I$) and $E$ and tangent to $Q_S$ at $A$ (resp. $I$). Note that the points $O_A$ and $O_I$ must be distinct, hence the line $(O_AO_I)$ is well-defined. Otherwise the circle $\calC_A=\calC_I$ would intersect the quartic curve $Q_S$ in five real points counted with multiplicity, which would contradict B\'ezout's theorem, since there are already intersection points with the circular points at infinity which both count twice. We have the following simple construction for $E_w'$ :

\begin{proposition}
\label{prop:tangentcircles}
The point $E_w'$ is obtained by reflecting $E$ through the line $(O_AO_I)$. In particular, the circle going through $A$, $E$ and $E_w'$ is tangent to $Q_S$ at $A$.
\end{proposition}

\begin{proof}

We distinguish two cases, whether $S$ is generic or trapezoidal. All the computations mentioned below are easily performed on a computer algebra software, but we do not display in this paper all the formulas obtained, since some of them would take up to ten lines.

\emph{Generic case.} Assume $S\in\calG$. Up to applying a similarity, one may assume that the equilateral hyperbola going through $B,D,E,F$ and $H$ has equation $xy=1$. Denote respectively by $b,d,e,f$ and $h$ the abscissas of $B,D,E,F$ and $H$. We will successively compute several quantities in terms of $b,d,e,f$ and $h$. We first compute the coordinates of $O_1$, $O_2$ and $O_4$, which are the respective circumcenters of the triangles $BDE$, $DEH$ and $EFH$. The points $A$ and $B$ are the two intersection points of the following two circles :
\begin{itemize}
 \item the circle centered at $O_1$ going through $B$ ;
 \item the image under the translation of vector $\overrightarrow{HB}$ of the circle centered at $O_2$ going through $H$.
\end{itemize}
The statement about the second circle follows from the fact that $S\in\calS_{2,2}$ has vertical monodromy equal to $\overrightarrow{BH}$. Denoting by $O_2'$ the image of $O_2$ under the translation of vector $\overrightarrow{HB}$, we obtain $A$ as the reflection of $B$ through the line $(O_1O_2')$ :
\[
A=(b+d+e,b^{-1}+d^{-1}+e^{-1})
\]
Applying translations of respective vectors $\overrightarrow{DF}$, $\overrightarrow{BH}$ and  $\overrightarrow{DF}+\overrightarrow{BH}$, we obtain the coordinates of the points $C,G$ and $I$ :
\begin{align*}
C&=(b+e+f,b^{-1}+e^{-1}+f^{-1}) \\
G&=(d+e+h,d^{-1}+e^{-1}+h^{-1}) \\
I&=(e+f+h,e^{-1}+f^{-1}+h^{-1}).
\end{align*}
Next, we compute the coordinates of $E_w$ (the reflection of $E$ through the line $(O_1O_4)$), $A_w$ (the reflection of $A$ through the line going through $O_1$ and $O_4+\overrightarrow{IA}$) and finally $E'_w=E_w+\overrightarrow{A_wA}$. Then we compute $\Omega$ (midpoint of $[AI]$) and the foci $P$ and $P'$ of the quartic curve $Q_S$ (which requires to first compute the respective circumcenters $O_B$ and $O_D$ of the triangles $ABC$ and $ADG$, as explained in Subsection~\ref{subsec:genericquartic}): 
\begin{align*}
\Omega&=(\frac{b+d+2e+f+h}{2},\frac{b^{-1}+d^{-1}+2e^{-1}+f^{-1}+h^{-1}}{2}) \\
P&=(\frac{b+d+e+f+h-bdefh}{2},\frac{b^{-1}+d^{-1}+e^{-1}+f^{-1}+h^{-1}-(bdefh)^{-1}}{2}) \\
P'&=(\frac{b+d+3e+f+h+bdefh}{2},\frac{b^{-1}+d^{-1}+3e^{-1}+f^{-1}+h^{-1}+(bdefh)^{-1}}{2})
\end{align*}
We also compute the real number
\[
\lambda=\frac{PA^2P'A^2-PC^2P'C^2}{\Omega A^2-\Omega C^2}.
\]
The quartic $Q_S$ then has an equation of the form
\[
Q_S=\left\{M\in\R^2 | PM^2P'M^2-\lambda \Omega M^2 = k \right\}
\]
for some real number $k$. Next, we compute the intersection point $O_A$ (resp. $O_I$) of the normal to the quartic at $A$ (resp. $I$) and the perpendicular bisector of the segment $[AE]$ (resp. $[IE]$). We finally check that $E'_w$ is indeed the reflection of $E$ through the line $(O_AO_I)$.

\emph{Trapezoidal case.} Assume $S\in\calT_h$. Up to applying a similarity, we may assume that the degenerate equilateral hyperbola going through $B,D,E,F$ and $H$ has equation $xy=0$ with $D,E$ and $F$ lying on the line $y=0$ and $B$ and $H$ lying on the line $x=0$. As in the generic case, we successively compute as functions of the coordinates of $B,D,E,F$ and $H$ the following quantities :
\begin{enumerate}
 \item the coordinates of $A,C,I$ and $E_w'$ ;
 \item the quantities $\alpha$ and $\beta$ using formulas~\eqref{eq:alpha} and~\eqref{eq:beta} ;
 \item the coordinates of $O_A$ and $O_I$.
\end{enumerate}
We finally check that $E'_w$ is indeed the reflection of $E$ through the line $(O_AO_I)$.

\end{proof}

\section{The group law on non-degenerate binodal quartics}
\label{sec:grouplaw}

We first prove Theorem~\ref{thm:quarticaddition} about the group law on general non-degenerate binodal quartics in Subsection~\ref{subsec:generalcase}, before applying it to prove Theorem~\ref{thm:translation} about Miquel dynamics in Subsection~\ref{subsec:Miquelcase}. The proof of Corollary~\ref{cor:invariantmeasure} is in Subsection~\ref{subsec:invariant}.

\subsection{General non-degenerate binodal quartics}
\label{subsec:generalcase}

Recall that two divisors on a Riemann surface are said to be linearly equivalent if their difference is the divisor of a meromorphic function. A \emph{complete linear system} is defined to be any set of all the divisors linearly equivalent to a given divisor. We will deal with divisors on the elliptic normalization $\hat\Ga$ of a non-degenerate binodal quartic $\Ga$.

\begin{lemma}
\label{lem:completesystem}
Let $\Ga\subset\cp^2$ be a non-degenerate binodal quartic curve with nodes $T_1$ and $T_2$. Consider the three-dimensional family $\calC$ of all the conics (including degenerations to unions of lines) passing through the nodes $T_1$ and $T_2$. A conic $c\in \calC$ intersects $\Ga$ at four additional points $X_c,Y_c,Z_c$ and $W_c$, which need not be distinct and may coincide with $T_1$ or $T_2$, and we denote by $\mcd_c$ the divisor $[X_c]+[Y_c]+[Z_c]+[W_c]$ on $\hat\Ga$. The set $\{\mcd_c\}_{c\in\calC}$ forms a complete linear system.
\end{lemma}

\begin{proof}
This lemma is an immediate consequence of the theory of adjoint curves developed for example in~\cite[section 49]{V45}. For an algebraic curve $\calG$ with singularities that are all nodes, an adjoint curve $\calA$ to $\calG$ is defined to be any curve passing through all the nodes of $\calG$. Let $\calD_{\calA}$ be the divisor corresponding to all the intersections of $\calA$ with $\calG$ outside of the nodes. The Brill-Noether residue theorem~\cite[p.216]{V45} states that the set $\{\calD_{\calA}\}$, where $\calA$ ranges over all the adjoint curves to $\calG$ of a fixed degree $d$, forms a complete linear system. The lemma corresponds to the special case with two nodes and adjoint curves of degree $d=2$. 
\end{proof}

We are now ready to prove Theorem~\ref{thm:quarticaddition}.

\begin{proof}[Proof of Theorem~\ref{thm:quarticaddition}]
For any conic $c\in\calC_P$, denote by $\calD_{c,P}$ the divisor $\calD_c-[P]=[X_c]+[Y_c]+[Z_c]$. The linear equivalence of two effective divisors on $\hat\Ga$ containing the base point $P$ is equivalent to the linear equivalence of their differences with the single-point divisor $[P]$, since linear equivalence classes of divisors form an additive group. Together with Lemma~\ref{lem:completesystem}, this implies that the set $\{\calD_{c,P}\}_{c\in\calC_P}$ forms a complete linear system.

For every two divisors $\sum_{j=1}^d[S_j]$ and $\sum_{j=1}^d[T_j]$ of the same degree $d$ on an elliptic curve the equality
\begin{equation}
\label{eq:pointssum}
\sum_{j=1}^d S_j=\sum_{j=1}^d T_j
\end{equation}
for the addition in the group law is independent on the choice of neutral element: if it holds for the group law defined by one neutral element, then it holds for every other neutral element. Equality~\eqref{eq:pointssum} holds if and only if the corresponding divisors are linearly equivalent. This is  a particular case of Abel's Theorem \cite[chapter 2, section 2]{GH78}.

Since $\{\calD_{c,P}\}_{c\in\calC_P}$ forms a complete linear system, for any $(c,c')\in(\calC_P)^2$, $[X_c]+[Y_c]+[Z_c]$ is linearly equivalent to $[X_{c'}]+[Y_{c'}]+[Z_{c'}]$, thus $X_c+Y_c+Z_c$ equals $X_{c'}+Y_{c'}+Z_{c'}$ for the group law on the elliptic normalization of $\Ga$, regardless of the choice of neutral element. By surjectivity of the tripling map, there exists $N$ such that $3N$ equals the quantity $X_c+Y_c+Z_c$, which is independent of $c\in\calC_P$. Taking this point $N$ as the neutral element, we obtain a group law on $\hat\Ga$ such that for any $c\in\calC_P$, $X_c+Y_c+Z_c=0$.
\end{proof}

\subsection{The group law on non-degenerate Miquel quartics}
\label{subsec:Miquelcase}

We first identify when a Miquel quartic is non-degenerate.

\begin{lemma}
\label{lem:quarticdegeneration}
A Miquel quartic $Q$ with an equation given by~\eqref{eq:Miquelquartic} is a non-degenerate binodal quartic curve if and only if $a\neq b$ and $4c \notin \left\{0,a^2, b^2\right\}$. In that case, its nodes are the two circular points at infinity.
\end{lemma}

\begin{proof}
The equation of $Q$ in homogeneous coordinates is
\begin{equation}
\label{eq:homogeneous}
(x^2+y^2)^2+ax^2z^2+by^2z^2+cz^4=0.
\end{equation}
A singular point of the quartic of homogeneous coordinates $(x:y:z)$ must satisfy~\eqref{eq:homogeneous} as well as the following three equations, corresponding to the vanishing of the three partial derivatives with respect to $x$, $y$ and $z$ of the left-hand side of~\eqref{eq:homogeneous}:
\begin{align}
x(2x^2+2y^2+az^2)&=0 \label{eq:partialx} \\
y(2x^2+2y^2+bz^2)&=0 \label{eq:partialy} \\
z(ax^2+by^2+2cz^2)&=0 \label{eq:partialz}
\end{align}
One checks that the two circulars points at infinity $T_1$ and $T_2$ of respective coordinates $(1:i:0)$ and $(1:-i:0)$ satisfy equations~\eqref{eq:homogeneous} to~\eqref{eq:partialz} for any choice of $a,b,c$. Furthermore, the Hessian of the left-hand side of~\eqref{eq:homogeneous} at $T_1$ is given by
\[
\left(
\begin{array}{ccc}
 8 & 8 i & 0 \\
 8 i & -8 & 0 \\
 0 & 0 & 2 (a-b) \\
\end{array}
\right),
\]
which has rank $2$ if and only if $a \neq b$. Hence $T_1$ is a node if and only if $a \neq b$. The same holds for $T_2$.

Equation~\eqref{eq:homogeneous} has no solution on the line at infinity $z=0$ besides $T_1$ and $T_2$. Distinguishing when $x$ or $y$ vanish, it is easy to check that for $a \neq b$, the only other singularities are
\begin{itemize}
 \item the point $(0:0:1)$ when $c=0$ ;
 \item the points $(\pm \sqrt{-a/2} :0 : 1)$ when $c=a^2/4$ ;
 \item the points $(0:\pm \sqrt{-b/2} : 1)$ when $c=b^2/4$.
\end{itemize}
This concludes the proof.
\end{proof}

We now use Theorem~\ref{thm:quarticaddition} to provide a geometric construction of the group law on a non-degenerate Miquel quartic.

\begin{proof}[Proof of Proposition~\ref{prop:miquelgrouplaw}]
Here $N$ is an intersection point of the non-degenerate Miquel quartic $Q_S$ with the $x$-axis. Since the nodes of $Q_S$ are the circular points at infinity and the (complex) circles are exactly the conics going through both circular points, the set $\calC_N$ consists in all the circles going through $N$. By symmetry with respect to the $x$-axis, the osculating circle to $Q_S$ at $N$ has an intersection of order $4$ with $Q_S$ at $N$. Thus $N$ can be taken as the neutral element for a group law on $Q_S$ with base point $N$ as constructed in Theorem~\ref{thm:quarticaddition}. Since both $Q_S$ and any tangent circle to $Q_S$ at $N$ are symmetric with respect to the $x$-axis, we deduce that the inverse for this group law is given by reflection through the $x$-axis. The statement about the sum of two points follows immediately.
\end{proof}

Theorem~\ref{thm:translation} is then an immediate consequence of Proposition~\ref{prop:tangentcircles}.

\begin{proof}[Proof of Theorem~\ref{thm:translation}]
By symmetry, it suffices to prove~\eqref{eq:whitemutation}. We consider two circles :
\begin{itemize}
 \item the osculating circle to $Q_S$ at $N$ ;
 \item the circle going through the points $A,E$ and $E_w'$, which is tangent to $Q_S$ at $A$ by Proposition~\ref{prop:tangentcircles}.
\end{itemize}
By Lemma~\ref{lem:completesystem}, the divisors $[E_w']+[E]+2[A]$ and $4[N]$ are linearly equivalent. For the group law on $Q_S$ with $N$ as base point and neutral element described in Proposition~\ref{prop:miquelgrouplaw}, we have $N=0$, hence $E_w'+E+2A=0$, by an argument similar to the one used in the proof of Theorem~\ref{thm:quarticaddition}.
\end{proof}

\subsection{The invariant measure on Miquel quartics}
\label{subsec:invariant}

\begin{proof}[Proof of Corollary~\ref{cor:invariantmeasure}]

Let $Q$ be a non-degenerate binodal Miquel quartic, with an equation of the form~\eqref{eq:Miquelquartic}. We show below that the pullback of the form $\omega$ to the elliptic normalization of the quartic $Q$ is a holomorphic differential. Firstly, $\omega$ is meromorphic.

Secondly, it has no poles at the intersection points of $Q$ with the coordinate axes. For example, consider a point $M\in\cc^2$ of its intersection with the $y$-axis. The germ at $M$ of the quartic is the graph of an even function $y=g(x)$, hence $g'(0)=0$ and $g(x)=g(0)+O(x^2)$ with $g(0)\neq0$, since $(0,0)\notin Q$ (which follows from the fact that $c\neq 0$ by Lemma~\ref{lem:quarticdegeneration}). Thus, in a neighborhood of the point $M$ one has $dy=O(x)dx$ and
\[
\omega=\frac{2dx}{y}+\frac{2dy}x=O(1)dx.
\]
The case of an intersection point with the $x$-axis is treated analogously.

Thirdly, we show below that the circular points at infinity are not poles of the restriction of the form $\omega$ to local branches of the quartic $Q$ at these points. There are two local branches of the quartic at each circular point, since each circular point is a node. Each of these local branches is regular and transverse to the infinity line, otherwise the intersection index of the quartic with the infinity line would be greater than $4$, which would contradict B\'ezout's theorem. Take an arbitrary local branch $\phi$ at a circular point, say $T_1$ with homogeneous coordinates $(1:i:0)$, and consider the restriction to it of the form $\omega$. The function $xy$ is meromorphic on $\cp^2$ with a pole of order $2$ along the infinity line and the local branch $\phi$ is transverse to it, thus the denominator $xy|_{\phi}$ has a pole of order $2$ at $T_1$. The primitive $(x^2+y^2)|_{\phi}$ of the numerator has at most first order pole at $T_1$, since the circular points satisfy the equality $x^2+y^2=0$. In more detail, let us introduce new affine coordinates $(u,v)$ on $\cc^2$ so that $x^2+y^2=uv$ and the $u$-axis intersects the infinity line at the point $T_1$. Define $\wt u=\frac1u$ and $\wt v=\frac vu$ and observe that they are local affine coordinates centered at $T_1$. The coordinate $\wt u$ can be taken as a local parameter of the branch $\phi$ since $\phi$ is transverse to the infinity line and $\wt u$ vanishes on the infinity line with order $1$. Furthermore, $\wt v$ is holomorphic and vanishes at $T_1$ thus $\wt v=O(\wt u)$ in a neighborhood of $T_1$ on $\phi$. Therefore, one has 
\[
x^2+y^2=uv=\frac{\wt v}{\wt u^2}=O(\frac1{\wt u})
\]
in a neighborhood of $T_1$ on $\phi$. Hence the restriction of the form $\omega$ to a local branch of $Q$ has a pole of order at most $2-2=0$ at a circular point. 

Thus the pullback of the form $\omega$ is a holomorphic differential on the normalization of $Q$. Recall that any holomorphic differential on an elliptic curve is invariant by any translation defined by the group structure on the curve. Its modulus induces a measure on the real part, which is invariant under every translation, hence under the translation induced on $E$ by the composition $\mu_b' \circ \mu_w'$. Note in passing that, since $\omega$ has no zeros, it induces a form of constant sign on each component of the real part of the quartic.
\end{proof}

\section*{Acknowledgements}
We thank Yuliy Baryshnikov for suggesting the study of the group law on Miquel quartics. We also thank Eugenii Shustin for pointing out the adjoint curve construction in~\cite{V45} and Ivan Izmestiev for drawing our attention to the connection between translation on binodal quartic curves and quadrilateral folding.

\label{Bibliography}
\bibliographystyle{plain}
\bibliography{bibliographie}

\end{document}